\documentclass[11pt,reqno]{amsart}
\usepackage[english]{babel}
\usepackage[utf8x]{inputenc}
\usepackage{amsmath}
\usepackage{amsthm}
\usepackage{graphicx}
\usepackage[colorinlistoftodos]{todonotes}
\usepackage{amsmath,amssymb,mathrsfs,amsthm,dsfont}
\usepackage[T1]{fontenc}
\usepackage{tikz}
\usetikzlibrary{matrix}
\title[From Moran to Wright Fisher]{When can the discrete Moran process may be replaced by Wright-Fisher diffusion?}
\author[G. Gorgui]{\textbf{\quad {Gorgui} GACKOU $^{\diamondsuit}$ \, \, }}
\address{{ {Gorgui} GACKOU}\\ Laboratoire de Math\'ematiques Blaise Pascal, CNRS UMR 6620, Universit\'e Clermont-Auvergne} \email{gorgui.gackou@uca.fr}
\author[A. Guillin]{\textbf{\quad {A.} Guillin $^{\diamondsuit}$ \, \, }}
\address{{ {Arnaud} GUILLIN}\\ Laboratoire de Math\'ematiques Blaise Pascal, CNRS UMR 6620, Universit\'e Clermont-Auvergne,
avenue des Landais, F-63177 Aubi\`ere.} \email{arnaud.guillin@uca.fr}
\author[A. Personne]{\textbf{\quad {Arnaud} Personne $^{\diamondsuit}$ \, \, }}
\address{{{Arnaud} PERSONNE}\\ Laboratoire de Math\'ematiques Blaise Pascal, CNRS UMR 6620, Universit\'e Clermont-Auvergne,
avenue des Landais, F-63177 Aubi\`ere.} \email{arnaud.personne@uca.fr}
\theoremstyle{definition}

\theoremstyle{plain}
\newtheorem{theorem}{Theorem}
\newtheorem{lemma}{Lemma}
\newtheorem{prop}{Proposition}

\newtheorem{remark}{Remark}

\newcommand{\pp}{\mathbb{P}}

\begin{document}

\maketitle

 \begin{center}

\textsc{$^{\diamondsuit}$ Universit\'e Clermont-Auvergne}
\smallskip

\end{center}


\begin{abstract}
The Moran discrete process and the Wright-Fisher model are the most popular models in population genetics. It is common to understand the dynamics of these models to use an approximating diffusion process, called Wright-Fisher diffusion. Here, we give a quantitative large population limit of the error committed by using the approximation diffusion in the presence of weak selection and weak immigration in one dimension. The approach is robust enough to consider the case where selection and immigration are Markovian processes, with limits jump or diffusion processes.
\end{abstract}

\section{Introduction}

The diffusion approximation is a technique in which a complicated and intractable (as the dimension increases) discrete Markovian process is replaced by an appropriate diffusion which is generally easier to study.
This technique is used in many domains and genetics and population dynamics are no exceptions to the rule. Two of the main models used in population dynamics are the Wright-Fisher (see for example \cite{fisher1},\cite{fisher2},\cite{Wright1},\cite{Wright2})
  and the Moran \cite{Moran} models which describe the evolution of a population having a constant size and subject to immigration end environmental variations. 
For large population limit, it is well known that the Moran process is quite difficult to handle mathematically and numerically. For example, the convergence to equilibrium (independent of the population size) or the estimation of various biodiversity index such as the Simpson index are not known. It is thus tempting to approach the dynamics of these Markovian process by a diffusion, called the Wright-Fisher diffusion, see for example \cite{ethier76}, \cite{EK86} or \cite{Kimura}, and work on this simpler (low dimensional) process to get  good quantitative properties. 

A traditional way to prove this result is to consider a martingale problem, as was developed by Stroock and Varadhan  in \cite{diffmultidim}, see also \cite{Dawson}, \cite{EK86} and \cite{EN89} for example for Wright-Fisher process with selection but without rates.
This technique ensures us that the discrete process converges to the diffusion when the size of the population grows to infinity. If the setting is very general and truly efficient, it is usually not quantitative as it does not give any order of the error committed in replacing the discrete process by the diffusion for fixed size of population. To obtain an estimation of this error we will consider another approach by Ethier-Norman in \cite{EthierNorman} (or \cite{errordiff}), which makes for a quantitative statement of the convergence of the generator using heavily the properties of the diffusion limit. For the Wright-Fisher model with immigration but without selection they showed that the error is of the order of the inverse f the population size, and uniform in time. Our main goal here will be to consider the more general model where 1) weak selection is involved; 2) immigration and selection may be also Markov processes. To include selection, constant or random is of course fundamental for modelization, see for example \cite{ewens2004}, \cite{Viotprocees}, \cite{kalyuzhny}, \cite{danino2016effect}, \cite{FOC2017}, \cite{danino2018fixation}, \cite{daninoshnerb} for recent references. Also, to study biodiversity, a common index is the Simpson index, which is intractable in non neutral model (see \cite{etienneOlff2} or \cite{etienne2005} in the neutral case, and even not easy to approximate via Monte Carlo simulation when the population is large. Based on the Wright-Fisher diffusion, an efficient approximation procedure has been introduced in \cite{pgj2018}. It is thus a crucial issue to get quantitative diffusion approximation result in the case of random selection to get a full approximation procedure for this biodiversity index.

Let us give the plan of this note. First in Section 2, we present the Moran model. As an introduction to the method, we will first consider the case of a constant selection and we find an error of the same order but growing exponentially or linearly in time. It will be done in section 3.
Sections 4 and 5 are concerned with the case of random environment. Section 4 considers the case when the limit of the selection is a pure jump process and Section 5 when it is a diffusion process. We will indicate the main modifications of the previous proof to adapt to this setting. An appendix indicates how to adapt the preceding proofs to the case of the Wright-Fisher discrete process.

\section{The Moran model and its approximation diffusion.}

Consider to simplify a population of $J$ individuals with only two species. Note that there no other difficulties than tedious calculations to consider a finite number of species. At each time step, one individual dies and is replaced by one member of the community or a member of a distinct (infinite) pool. To make precise this mechanism of evolution, let us introduce the following parameters:
\begin{itemize}
\item $m$ is the immigration probability, i.e. the probability that the next member of the population comes from the exterior pool;
\item $p$ is proportion of the first species in the pool;
\item $s$ is the selection parameter, which acts at favoring one of the two species.
\end{itemize}
Let us first consider that $m$, $p$ and $s$ are functions depending on time (but not random to simplify) and taking values in $[0,1]$ for the first two and  in $]-1;+\infty[$ for the selection parameter.\\
Note that this process may also be described considering mutation, rather than immigration but there is a one to one relation between these two interpretations. Our time horizon will be denoted by $T$.\\
Rather than considering the process following the number of elements in each species, we will study the proportion in the population of the first species. To do so, let $I_{J}=\lbrace \frac{i}{J}: i=0,1,2,\cdots,J\rbrace $, and we denote for all $f$ in $B(I_J)$, the bounded functions on $I_{J}$, $$\Vert f\Vert_{J}=\max\limits_{x\in I_{J}}|f(x)|.$$
Consider also $\Vert g\Vert=\sup|g|$ the supremum norm of $g$.\\
Let $X_{n}^{J}$, with values in $ I_{J} $, be the proportion of individuals of the first species in the community.\\
In this section, $X_{n}^{J}$ is thus the Moran process, namely a Markov process evolving with the following transition probabilities: denote $\Delta=\frac{1}{J}$\\

\begin{equation*}
  \left\{
    \begin{aligned}
      \pp(X^J_{n+1}=x+\Delta|X^J_{n}=x)&=(1-x)\left(m_{n}p_{n}+(1-m_{n})\frac{x(1+s_{n})}{1+xs_{n}}\right)\\&:=P_{x+}\\
       \pp(X^{J}_{n+1}=x-\Delta|X^J_{n}=x)&=x\left(m_{n}(1-p_{n})+(1-m_{n})\Big(1-\frac{x(1+s_{n})}{1+xs_{n}}\Big)\right)\\ &  :=P_{x-}\\
       \pp(X^{J}_{n+1}=x|X^J_{n}=x)&=1-P_{x+}-P_{x-}.
    \end{aligned}
  \right.
\end{equation*}

To study the dynamical properties of this process a convenient method developped first by Fisher \cite{fisher1}, \cite{fisher2} and then Wright \cite{Wright1}, \cite{Wright2},  aims at approximating this discrete model by a diffusion when the size of the population tends to infinity.\\

In the special case of the Moran model with weak selection and weak immigration, meaning that the parameters $s$ and $m$ are inversely proportional to the population size $J$, we usually use the process
 $\lbrace Y_{t}^{J}\rbrace_{t \geq 0}$
 taking values in  $I=[0,1]$  
 defined by the following generator: 
$$ L=\frac{1}{J^{2}}x(1-x)\frac{{\partial}^2}{\partial{x}^2}+ \frac{1}{J}[ sx(1-x)+m(p-x)]\frac{\partial}{\partial{x}}.$$

Note that, in weak selection and immigration, $s=s'/J$ and $m=m'/J$, so the process defined by $\lbrace Z_{t}\rbrace_{t \geq 0}=\lbrace Y_{J^{2}t}^{J}\rbrace_{t \geq 0}$ do not depend on $J$. Its generator is $$ \mathscr{L}=x(1-x)\frac{{\partial}^2}{\partial{x}^2}+ [ s'x(1-x)+m'(p-x)]\frac{\partial}{\partial{x}}$$
or equivalently by the stochastic differential equation
$$dZ_t=\sqrt{2Z_t(1-Z_t)}dB_t+\left[s'Z_t(1-Z_t)+m'(p-Z_t)\right]dt.$$ 

Our aim is to find for sufficiently regular test function, say $ f \in \mathscr{C}^{4} $, an estimation of :
\begin{align*}
\left\|\mathbb{E}_x\Big[f(Z_{[t]})\Big]-\mathbb{E}_x\Big[f(X_{[J^{2}t]}^{J})\Big] \right\|_J
\end{align*}
for $0\leq t \leq T $  and for all $ x $ in $ I_J$. By replacing $ Z_{t}$ by $Y_{J^{2}t}^{J} $ we thus get :
\begin{align*}
\left\|\mathbb{E}_x\Big[f(X_{[J^{2}t]}^{J}) \Big]-\mathbb{E}_x\Big[f(Y_{[J^{2}t]}^{J}) \Big] \right\|_J
\end{align*}
 for $ 0 \leq t \leq  T $ , and $ x \in I_J $. \\
So equivalently it is convenient to study, if we note $ n=[J^{2}t]$ :
\begin{align*}
\Vert\mathbb{E}_x\Big[f(X_{n}^{J} )\Big]-\mathbb{E}_x\Big[f(Y_{n}^{J}) \Big] \Vert_J
\end{align*}
 on  $ 0 \leq n \leq J^2T $ ,and  $ x \in I_J $.

\section{Estimate of the error in the approximation diffusion for constant weak immigration and selection}

\subsection{Main result}
We now give our main result in the case where immigration and selection are constant. It furnishes an estimation of the error committed during the  convergence of the discrete Moran process $X_{n}$ toward the Wright-Fisher diffusion process $Y_{n}$. 
\begin{theorem}
Let us consider the weak immigration and selection case, so that $s=\frac{s'}J$ and $m=\frac{m'}J$ for some $s'\in\mathbb{R}$, $m'\in\mathbb{R}_+$ ($J$ large enough). Let $p\in]0,1[$. Let $f \in C^{4}(I) $  then there exist positive $a$ and $b$ (depending on $m'$ and $s'$)  such that:
\begin{align*}
\left\| \mathbb{E}_x \Big[f(X_{n}^{J} )\Big]-\mathbb{E}_x \Big[f(Y_{n}^{J} ) \Big] \right\|_J \le\left(\Vert f^{(1)} \Vert_{J} +\Vert f^{(2)} \Vert_{J}\right)\frac{ae^{bn}}{J}+o(\frac{1}{J}) .
\end{align*}
If we suppose moreover that $m'>|s'|$ then there exists $a>0$ such that
\begin{align*}
\left\|\mathbb{E}_x \Big[f(X_{n}^{J} )\Big]-\mathbb{E}_x \Big[f(Y_{n}^{J} ) \Big] \right\|_J \le\left(\Vert f^{(1)} \Vert_{J} +\Vert f^{(2)} \Vert_{J}\right)\frac{an}{J}+o(\frac{1}{J}) .
\end{align*}
\end{theorem}

\begin{remark}
By considering $s=0$ then $b=0$ and we find back the uniform in time approximation diffusion with speed $J$. Our method of proof, requiring the control of some Feynman-Kac formula based on the limiting process, seems limited to give non uniform in time result. Our hope is that we may get weaker conditions than $m>|s|$ to get linear in time estimates. Another possibility is to mix these dependance in time approximation with known ergodicity of the Wright-Fisher process, as in Norman \cite{norman77}.
\end{remark}
\begin{remark}
We have considered to simplify $s=\frac{s'}J$ and $m=\frac{m'}J$ but one may generalize a little bit the condition to locally bounded $s$ and $m$ such that $\lim Js<\infty$ and $\lim Jm<\infty$.
\end{remark}
\begin{remark}
Such approximation error is noticeably useful to polynomial test function $f$, so that we may for example consider the Simpson index of the Moran process, see \cite{pgj2018} for further details.
\end{remark}

\begin{remark}
The following figures show that the obtained rate $\frac1J$ is of the good order.

\begin{figure}[htbp]
\begin{minipage}[c]{.45\linewidth}
\begin{center}
\includegraphics[scale=0.35]{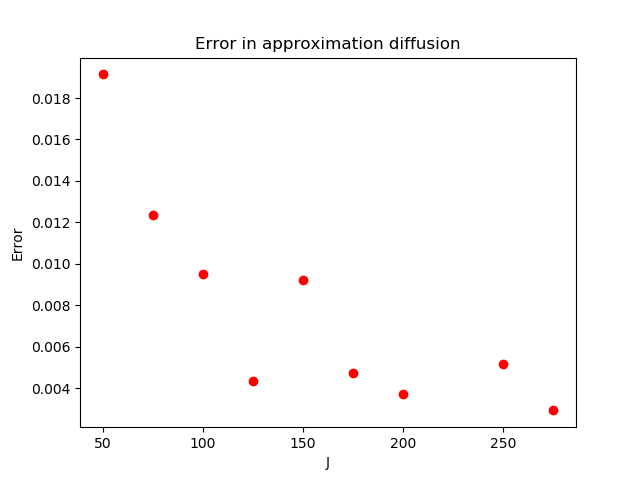}
\end{center}
\end{minipage}
\hfill
\begin{minipage}[c]{.45\linewidth}
\begin{center}
\includegraphics[scale=0.35]{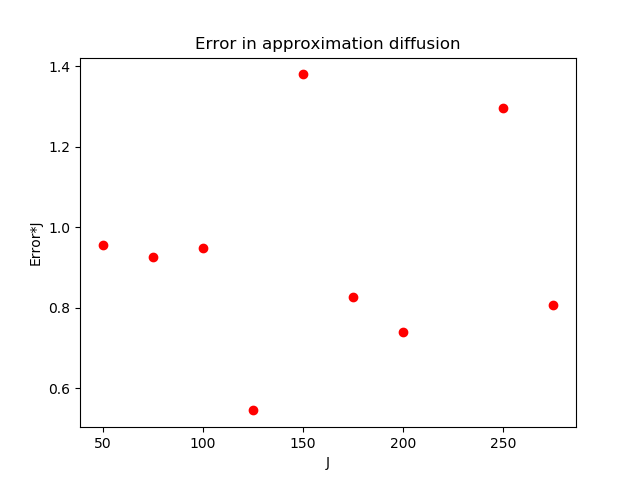}
\end{center}
\end{minipage}
\label{fig:image1}
\caption{Conditions: $s=1$, $m=0.2$, $p=0.5$, $X_0=0.7$. Left hand side : Monte Carlo estimation of the error in the approximation, using $f(x)=x$. Right hand side: same Monte Carlo estimation of the error times $J$.}
\end{figure}
It shows that our rate may be the good one.
\end{remark}

\subsection{Proof}

The proof relies on three ingredients: 
\begin{enumerate}
\item a "telescopic" decomposition of the error; \item a quantitative estimate of the error at time 1 of the approximation of the Moran process by the diffusion; \item quantitative control of the regularity of the Wright-Fisher process. \end{enumerate}
Note also that in the sequel we will not make distinction between function on ${\mathscr C}(I)$ and their restrictions on $I_J$. \\

Let $ S_n $ be defined on   $ B(I_J)$ (the space of bounded functions on $ I_J $) by :
$$ (S_nf)(x)=\mathbb{E}_x[f(X_{n}^J)] \quad \forall n \in  \mathbb{N}.$$
As is usual $S_{n}$ verifies for all $k$ in  $\mathbb{N}$ the semigroup property, namely that  $S_{n+k}=S_nS_k$.\\

Let $T_{t}$ be the operator defined on the space of bounded continuous function by :
$$ (T_tf)(x)=\mathbb{E}_x[f(Y_{t}^{J})] \quad \forall t \geq 0.$$
It also defines a semigroup $ T_{s+t}=T_sT_t $, $ \forall s \geq  0. $

Thanks to these properties, we have

\begin{align*}
S_nT_0f-S_0T_nf&=\sum_{k=0}^{n-1}S_{n-k}T_kf-S_{n-k-1}T_{k+1}f\\ 
S_nT_0f-S_0T_nf&=\sum_{k=0}^{n-1}S_{n-k-1}S_1T_kf-S_{n-k-1}T_1T_kf\\&=\sum_{k=0}^{n-1}S_{n-k-1}(S_1-T_1)T_kf 
\end{align*}

and as $ \Vert S_nf \Vert_J \leq  \Vert f \Vert_J $ by triangular inequality, we get $ \forall n \in \mathbb{N}$, $\forall f \in \mathscr{C}(I_J) $.

\begin{align}
\Vert S_nf -T_nf \Vert_J \leq \sum_{k=0}^{n-1} \Vert S_{n-k-1}(S_1-T_1)T_kf \Vert_J \leq \sum_{k=0}^{n-1} \Vert (S_1-T_1)T_kf \Vert_J .\label{equation11}
\end{align}

We have two main terms to analyze : $S_1-T_1$ for a "one-step" difference between the Moran process and the Wright-Fisher diffusion process, and $T_kf$ for which we need regularity estimates.

\subsubsection*{Control of $(S_1-T_1)$}$\,$\\
Let us first study, for $f$ regular enough , $ (S_1-T_1)f$. The main goal is to obtain the Taylor expansion of this function when $J$ is big enough.

\begin{lemma}\label{lemme1}
When $J$ is big enough, i.e. $s=s'/J>-1+\varepsilon$, there exists $K_1(\varepsilon)>0$ such that
$$ \Vert(S_1-T_1)f\Vert_{J}\leq\frac{|s'|m'+s'^2}{4\varepsilon J^3} \Vert f^{(1)} \Vert_{J} +\frac{m'p+\frac{|s'|}{4}+m'(1+|s'|)}{2\varepsilon J^3} \Vert f^{(2)} \Vert_{J}+\frac{K_{1}(\varepsilon)}{J^4}$$
\end{lemma}

\begin{proof}
Let us begin by consideration on the Wright-Fisher diffusion process. Remark first, as usual for this diffusion process
$$\lim_{t \rightarrow 0} \left\|  \frac{T_tf-f}{t}- Lf \right\|_J=0 ,\forall f \in \mathscr{C}^2(I).$$
The Chapman-Kolmogorov backward equation reads
$$\frac{\partial}{\partial t}(T_tf)(x)=L(T_tf)(x)=T_t(Lf)(x)$$
and more generally if $f$ is enough regular, for $j$ in $\mathbb{N}$ it is possible to define $L^j$ as:
$$
\frac{\partial^{j}}{\partial t^j}(T_tf)(x)=(T_tL^jf)(x),\forall x \in I, t \geq 0.$$
For this proof, we only need to go to the fourth order in $j$. So let $f\in\mathscr{C}^4(I)$ (possibly depending on $J$),  using Taylor theorem for $ (T_1f)(x) $ there exists $w_{2}$, independent of $J$,  such as :
\begin{eqnarray} 
(T_1f)(x)&=&(T_0f)(x)+(T_0 f)^{(1)}(x)(1-0)+w_{2}\frac{1}{2!}(T_0f)^{(2)}(x)\nonumber\\ 
&=&f(x)+(L^{1}f)(x)+\frac{w_{2}}{2}(L^{2} f)(x)\label{TaylorT} 
\end{eqnarray}
By direct calculations, we have for the successive $L^j$
\begin{align*}
L^{1}f(x)&=\frac{x(1-x)}{J^2}f^{(2)}(x)+\frac{sx(1-x)+m(p-x)}{J}f^{(1)}(x)\\
L^{2}f(x)&=\frac{(x(1-x))^2}{J^4}f^{(4)}(x)\\
&+ \Big[ \frac{2(1-2x)x(1-x)}{J^4}+\frac{2x(1-x)(sx(1-x)+m(p-x))}{J^3}\Big]f^{(3)}(x)\\
&+ \Big[   \frac{-2x(1-x)}{J^4} \\
&+\frac{2x(1-x)(s(1-2x)-m)+(1-2x)(sx(1-x)+m(p-x))}{J^3}\\
&+\frac{(sx(1-x)+m(p-x))^2}{J^2}\Big]f^{(2)}(x) + \Big[\frac{-2sx(1-x)}{J^3}\\
&+\frac{(s(1-2x)-m)(sx(1-x)+m(p-x))}{J^2}\Big]f^{(1)}(x)\\
\end{align*}
where $f^{(i)}$ the $i^{th}$ derivative of $f$. Remark now that by our assumption on the boundedness of the successive derivatives of $f$ that there exists $K_0$ (depending also on $m',p,s'$)
$$\|L^2f\|_J\le\frac{K_{0}}{J^4}$$
Thus, in the following this term could be neglected.

 Let us now look at the Moran process and so get estimates on $S_1$. The quantity $X_1-x$ is at least of the order of $\frac{1}{J}$ and when $J$ goes to infinity, goes to $0$.  So using Taylor's theorem, there exists $|w_{1}|<1$ such that :
$$f(X_1)=f(x) +f^{(1)}(x)(X_1-x)+\frac{f^{(2)}(x)}{2}\big(X_1-x\big)^{2}+\frac{w_1}{3!}(X_1-x)^3 \Vert f^{(3)} \Vert_{J} $$
and thus
 \begin{equation}
 \begin{aligned}
(S_1f)(x)&=\mathbb{E}_x[f(X_1)]\\
&=f(x) +f^{(1)}(x)\mathbb{E}_x[X_1-x]\label{TaylorS}+ \frac{1}{2!}\mathbb{E}_x[(X_1-x)^2]  f^{(2)}(x)\\  
&\quad+\frac{w_1}{3!}\mathbb{E}_x[(X_1-x)^3] \Vert f^{(3)} \Vert_{J}  
\end{aligned} 
 \end{equation}
Direct estimates (even if tedious) on the centered moments of the Moran process give
\begin{align}
&\mathbb{E}_x[X_1-x]=\frac{sx(1-x)(1-m)}{J(1+sx)}+\frac{m(p-x)}{J} \label{esperance1}\\ 
&\mathbb{E}_x[(X_1-x)^2]=\frac{1}{J^2} \Big (mp(1-2x)+\frac{(1-m)(1+s)x(1-2x)}{1+sx}+x\Big) \\
&\mathbb{E}_x[(X_1-x)^3]<K_{3}\frac{1}{J^4}
\end{align}
where $K_{3}$ is a constant (independent of $J$).

We may then consider $(S_1-T_1)f$ through \eqref{TaylorS} and \eqref{TaylorT}  so that there exists a constant $K_{1}$  such as :
\begin{align}   
(S_1f)(x)-(T_1f)(x)=&f^{(1)}(x)\mathbb{E}_x[(X_1-x)]+\frac{1}{2}\mathbb{E}_x[(X_1-x)^2]  f^{(2)} (x)  \nonumber\\
+&\frac{w_1}{6}\mathbb{E}_x[(X_1-x)^3] \Vert f^{(3)} \Vert_{J} 
-\Big( (L^1f)(x)+\frac{w_{2}}{2}(L^2f)(x) \Big) \nonumber \\
=&\gamma^J_1f^{(1)} (x) +\gamma^J_2 f^{(2)} (x)+\frac{K_{1}}{J^4} \label{gamma}
\end{align}
with
\begin{align*}\nonumber
\gamma^J_{1}&=\frac{-sx(1-x)(m+sx)}{J(1+sx)}\\
\gamma^J_2&=\frac{1}{2J^2}\left [mp(1-2x) +\frac{(1-m)(1+s)x(1-2x)}{1+sx}
-x(1-2x) \right]\\
&=\frac{1}{2J^2}\left [mp(1-2x) +x(1-2x)\frac{s(1-x)-m(1+s)}{1+sx} \right]
\end{align*}
As selection and immigration are weak, we easily conclude that $\gamma^J_{1} $ and $ \gamma^J_{2}$ are at most of the order of $\frac{1}{J^3}$.
\end{proof}

\subsubsection*{Regularity estimates on $T_t$}$\,$\\
We have now to prove regularity estimates on $T_t$. By \cite[Th.1]{ethier76}, we have that $ T_t:\mathscr{C}^{j}(I)\rightarrow \mathscr{C}^j(I) $ for all $j$. Assume for now that  $f \in \mathscr{C}^4(I) $ and $ \forall j\in\{1,2\}$, 
$\forall k\leqslant j$, there are $ c_j$ and $a_{k,j}$ $\in\mathbb{R}^{+}$ independent of $J$ such that:
 \begin{equation}\label{crucest}
   \Vert (T_tf)^{(j)} \Vert_J \leq  e^{c_j\frac{t}{J^2}}\,\sum\limits_{i=1}^{j} a_{i,j}\Vert f^{(j)} \Vert_J \end{equation}
   with $c_{j}=\sup\limits_{x\in[0,1]}|j(j-1)-Js(1-2x)-Jm | $.
   
Let us see how to conclude if \eqref{crucest} is verified. First it exists a continuous (time dependent) function $\Tilde{R}_{j}$ independent of $J$ such that:
\begin{align*}
\Vert S_nf-T_nf \Vert_{J} &\leq \sum_{k=0}^{n-1}(\Vert (S_1-T_1)T_kf \Vert_{J} \\
&\leq \sum_{k=0}^{n-1}\left(\Vert\gamma_1\Vert \Vert (T_kf)^{(1)} \Vert_{J} +\Vert\gamma_2\Vert \Vert (T_kf)^{(2)} \Vert_{J}+O\left(\frac{1}{J^4}\right)\right)  \\
& \leq \sum_{j=1}^{2}\Vert\gamma^J_j \Vert \sum\limits_{i=1}^{j} a_{i,j}\Vert f^{(j)} \Vert_{J}\sum_{k=0}^{n-1}\exp(c_j\frac{k}{J^2})+O\left(\frac{1}{J^2}\right) \\
&\leq  \sum_{j=1}^{2}\Vert\gamma^J_j\Vert \times\Tilde{R}_{j}J^2 \times\Vert f^{(j)} \Vert_{J}+O\left(\frac{1}{J^2}\right )  \underset{J\to+\infty}{\longrightarrow}0
\end{align*}because if $J$ is big enough,
\begin{align*}
\sum\limits_{i=1}^{j} a_{i,j}\sum_{k=0}^{n-1}\exp\left(c_j\frac{k}{J^2}\right) =\sum\limits_{i=1}^{j} a_{i,j}\frac{1-\exp(c_{j}\frac{n}{J^2})}{1-\exp(\frac{c_j}{J^2})}\leq J^2\Tilde{R}_{j}(t)
\end{align*}
and $\Tilde{R}_{j}$ is independent of $J$ because $n$ is  of the order of $J^2$.
And so with $q(t)=\max\limits_{j\in{1,2}}\Big(\Vert\gamma_j\Vert  \Tilde{R}_{j}J^3 \Big) $, we obtain the result:
$$\Vert S_nf-T_nf \Vert_{J} \leq \frac{q(t) }{J}\Big(\Vert f^{(1)} \Vert_{J}+\Vert f^{(2)} \Vert_{J}\Big)+O\left(\frac{1}{J^2}\right ).$$
This concludes the proof in the first case. Indeed, we easily see that the function $\Tilde{R}_j(t)$ is exponential in time in the general case. We will see later how, when some additional conditions are added on $m'$ and $s'$, one may obtain a linear in time function.\\
We will now prove the crucial \eqref{crucest}. It will be done through the following proposition.
\begin{prop}\label{prop1}
Let $ \phi(t,x)=(T_tf)(x) $,  $ x \in I_J$ and $t \geq 0 $.
Assume $ f\in \mathscr{C}^{j+2}(I) $ then  $ \phi(t,x) \in \mathscr{C}^{j+2}(I) $  and  for $  j\in \mathbb{N}$, $\forall k\leqslant j$, there are $ {c}_j$and $a_{k,j}$ $\in\mathbb{R}$ independent of $J$ such as  $ \Vert  \phi(t,x)^{(j)} \Vert \leq  \exp({c}_j\frac{t}{J^2} )\sum\limits_{k=1}^{j} a_{k,j}\Vert f^{(j)} \Vert $\\
\end{prop}
\begin{proof}
First remark that  the Chapman-Kolmogorov backward equation may be written :
$$\frac{\partial}{\partial t}\phi=L\phi, \qquad\qquad
 \phi(0,.)=f.$$
The following lemma gives the equations verified by $\frac{\partial}{\partial t}\phi^{(j)}$:
\begin{lemma}
\label{lemme2}
Let  $ \phi^{(j)}$ be the $j^{th}$derivative of  $\phi $  with respect to $x$ then we get:
$$\frac{\partial}{\partial t}\phi^{(j)}=L_j\phi^{(j)}-\nu_j\phi^{(j)}+\psi_{j}\phi^{(j-1)},\qquad\qquad \phi^{(j)}(0,.)=f^{(j)}$$
where 
\begin{align*}
&L_j\phi^{(j)}(x)=L\phi^{(j)}+j\frac{1-2x}{J^2}\phi^{(j+1)}\\
&\nu_j(x)=\left(\frac{j(j-1)}{J^2}-j\frac{s(1-2x)-m}{J}\right)\\\
&\psi_{j}=\frac{-sj(j-1)}{J}.
\end{align*}
\end{lemma}
Let us remark that there are two new terms when there is selection in Moran processes, i.e. $\psi_j$ which will lead to the dependence in time of our estimates handled via Feynman-Kac formula, and one in $\nu_j$ which will be the key to the condition to get only linear in time dependence.
\begin{proof}
\quad\\
A simple recurrence is sufficient to prove this result, for simplicity let us only look at the case  j=1 
\begin{align*}
\frac{\partial}{\partial t}\phi^{(1)}&=\frac{\partial}{\partial x}\left(\frac{\partial}{\partial t}\phi\right)=\frac{\partial}{\partial x}L\phi \\
&=\frac{\partial}{\partial x} \Big(\frac{x(1-x)}{J^2}\phi^{(2)}+\frac{sx(1-x)+m(p-x)}{J}\phi^{(1)}                    \Big) \\
&=\frac{x(1-x)}{J^2}\phi^{(3)}+\frac{1}{J}(sx(1-x)+m(p-x))\phi^{(2)}+\frac{1-2x}{J^2}\phi^{(2)}\\
&\qquad+(s(1-2x)-m)\frac{1}{J}\phi^{(1)}\\
&= (L+\frac{1-2x}{J^2}\frac{\partial}{\partial x })\phi^{(1)}+\frac{s(1-2x)-m}{J} \phi^{(1)}\\
&=L_1\phi^{(1)}+\frac{s(1-2x)-m}{J}\phi^{(1)}.
\end{align*}
With $L_1\phi^{(1)}=L\phi^{(1)}+\frac{1-2x}{J^2}\frac{\partial\phi^{(1)}}{\partial x }$, we find the good initial coefficients.
\end{proof}
Let us now use the  Feymann-Kac formula to get ,
\begin{equation}
\begin{aligned}
\phi^{(j)}(t,x)&=\mathbb{E}_x\left[f^{(j)}(Y^{j}_t)\exp\left(-\int_{0}^{t}\frac{j(j-1)}{J^2}+\frac{m-s(1-2Y^{j}_{u})}{J}du\right)\right. \\ \nonumber
&\qquad-\left.\int\limits_{0}^{t}\frac{sj(j-1)}{J}\phi^{(j-1)}(Y^{j}_{h})e^{-\int_{0}^{h}\frac{j(j-1)}{J^2}+\frac{m-s(1-2Y^{j}_{u})}{J}du}dh\right] \label{feykack}
\end{aligned}
\end{equation}
with $Y^{j}_t$ the process having $L_{j}$ for generator. Then  look first at $j=1$. As we are in weak selection and weak immigration, 
\begin{align*}
\Vert\phi^{(1)} (x,t)\Vert &\leq \mathbb{E}_x\left[\Vert f^{(1)} (\tilde{Y}_t)\Vert\exp\left(\frac{t}{J^2} \sup\limits_{x\in [0,1]}J(m-s(1-2x)\right) \right]\\
&\leq \Vert f^{(1)} \Vert \exp\left(\frac{t}{J^2} \lambda_{1}\right)
\end{align*}
where $\lambda_{1}=\sup\limits_{x\in [0,1]}J(m-s(1-2x))= m'+|s'|$ is independent of $J$ . The case $j=1$ is proved. \\

We will then prove the result by recurrence: suppose true this hypothesis until $j=j-1$.\\
For $j>1$, denote $c_{j}=\sup\limits_{x\in[0,1]}|J^2\nu_{j}(x)|$,  and remark that $c_{j}$ is no equal to zero and is independent of $J$ because the selection and immigration are weak.  Thus
\begin{align*}
\Vert\phi^{(j)}(t,x)\Vert&\leq\mathbb{E}_x\Big[| f^{(j)}(Y^{j}_t)| e^{c_{j}\frac{t}{J^2}}+\int\limits_{0}^{t}\frac{sj(j-1)}{J}\Vert \phi^{(j-1)}(Y^{j}_{h})\Vert e^{\frac{hc_{j}}{J^2}}dh\Big]\\
&\leqslant \Vert f^{(j)}(x)\Vert e^{c_{j}\frac{t}{J^2}}+\frac{sj(j-1)}{J}\Vert \phi^{(j-1)}(x)\Vert_{\infty} \int\limits_{0}^{t}e^{\frac{hc_{j}}{J^2}}dh\\
&\leqslant \Vert f^{(j)}(x)\Vert e^{c_{j}\frac{t}{J^2}}+\frac{sj(j-1)}{J}\Vert \phi^{(j-1)}(x)\Vert_{\infty} \frac{J^2}{c_{j}}\big(e^{\frac{tc_{j}}{J^2}}-1\big)\\
&\leqslant      \Vert f^{(j)}(x)\Vert e^{c_{j}\frac{t}{J^2}}\\
&\qquad + \exp({c}_{j-1}\frac{t}{J^2} )\sum\limits_{k=1}^{j-1} a_{k,j-1}\Vert f^{(j)} \Vert \frac{Jsj(j-1)}{c_{j}}\Big(e^{\frac{tc_{j}}{J^2}}-1\Big)\\
&\leqslant e^{\frac{tc_{j}}{J^2}}\left(  \Vert f^{(j)}(x)\Vert+  e^{{c}_{j-1}\frac{t}{J^2} }\sum\limits_{k=1}^{j-1} a_{k,j-1}\Vert f^{(j)} \Vert \frac{Jsj(j-1)}{c_{j}}\right)\\
&\leqslant e^{\frac{t{c}_{j}}{J^2}}\left( \sum\limits_{k=1}^{j} a_{k,j}\Vert f^{(j)} \Vert\right).
\end{align*}
The $a_{k,j}$ do not depend on $J$, because  $\frac{Jsj(j-1)}{c_{j}}$ , the $a_{k,j-1}$ and $\exp(\lambda_j\frac{t}{J^2} )$ can be bounded independently of $J$.\\

To conclude we have to justify that  $c_j$ is finite for all $j$.  For it we just need to note  that the processes $Y_{t}^{j}$ are bounded by $0$ and $1$ for all $j$. \\
This is partly due to the fact that their generator $L_j\phi^{(j)}(x)=L\phi^{(j)}+j\frac{1-2x}{J^2}\phi^{(j+1)}$ has a negative drift at the neighbourhood of $1$ and a positive at the  neighbourhood of $0$, see Feller\cite{Feller2}. This argument completes the proof.
\end{proof}

Let us now consider the case where $m>|s|$, we will show in this case that we obtain a  linear in time dependance rather than an exponential one. Then, in the equation \eqref{feykack} we can use the following:
 \begin{align*}
\Vert\phi^{(1)} (x,t)\Vert \leq\Vert f^{(1)} \Vert \exp(-\frac{t}{J^2} \lambda_{1}) \\
\Vert\phi^{(2)} (x,t)\Vert \leq c_{1}\left(\Vert f^{(1)} \Vert+\Vert f^{(2)} \Vert\right) 
 \end{align*}where $c_{1}$ is a constant independent of time.
 And then,
 \begin{align*}
\Vert S_nf-T_nf \Vert &\leq \sum_{k=0}^{n-1}\Vert (S_1-T_1)T_kf \Vert \\
&\leq \sum_{k=0}^{n-1}\left(\|\gamma_1^{J}\| \Vert (T_kf)^{(1)} \Vert +\|\gamma_2^{J} \|\Vert (T_kf)^{(2)} \Vert+O\left(\frac{1}{J^4}\right)\right)  \\
&\leq \| \gamma_1^{J}\|\, \Vert f^{(1)} \Vert \sum_{k=0}^{n-1}exp(-\frac{k}{J^2}\times \lambda_{1})+\|\gamma_2^{J} \|\,c_{1}\left(\Vert f^{(1)} \Vert+\Vert f^{(2)} \Vert\right) n \\
& +O\left(\frac{1}{J^2}\right) \\
&\leq \max(\|\gamma^J_{1}\|,\|\gamma_{2}^J\|)\times J^2c(t+1)\left(\Vert f^{(1)} \Vert+\Vert f^{(2)} \Vert\right)+O\left(\frac{1}{J^2}\right) 
\end{align*}because if $J$ is big enough,
\begin{align*}
\sum_{k=0}^{n-1}\exp(-\frac{k}{J^2}\times \lambda_{1})=\frac{1-\exp(-\frac{n}{J^2}\times \lambda_{1})}{1-\exp(-\frac{\lambda_{1}}{J^2})}\leq c_{2}J^2
\end{align*}
and $c=max(c_{1},c_{2})$ is independent of $J$ and independent of time.

\section{Random limiting selection as a pure jump process}

To simplify, we will consider a constant immigration, in order to see where the main difficulty arises. The results would readily apply also to this case.\\
Let us now assume that $s$ is no longer a constant but a Markovian jump process $(s_{n})_{n\in \mathbb{N}}$ with homogeneous transition probability $(P_{x,y})$.
 We are in the weak selection case so $s_{n}$ is still of the order of $\frac{1}{J}$ and takes values in a {\it finite} space $E$.\\
Assume furthermore 
\begin{equation}
P^{J}_{s,s'}\times J^2 \underset{J\to+\infty}{\longrightarrow}\alpha_{s}Q_{s,s'}\qquad\qquad\forall s \neq s' \label{assumption1}.
\end{equation}
As in the previous section, $(X_{n})_{n}$ is the Moran process, but with a Markovian selection and
$(X_{n})_{n}$ takes values in $I_{J}$. Finally denote $\tilde{Z}_{n}=(X_n,s_n)$. Consider now the processes $Z_{t}$ tacking values in $\mathscr{I}=[0.1]\times E$ defined by the following generator:
\begin{align*}
 L_{x,s}f(x,s)=&\frac{1}{J^{2}}x(1-x)\frac{{\partial}^2}{\partial{x}^2}f(x,s)+ \frac{1}{J}[ sx(1-x)+m(p-x)]\frac{\partial}{\partial{x}}f(x,s)\\
 &+\sum\limits_{s'\in E}\frac{\alpha(s)Q_{s,s'}}{J^2}\big(f(x,s')-f(x,s)\big)\quad  \forall f \in \mathscr{C}²(\mathscr{I})
 \end{align*}

Its first coordinate is the process $Y_{t}$ having the same generator as in the first part and the second is $S_{t}$ the Markovian jump process having $(Q_{s',s})_{s,s'\in E}$ for generator and $\frac{\alpha}{J^2}$ for transition rates.\\
As in the previous part we want to quantify the convergence of $\tilde{Z}_{n}$ towards $Z_{n}$ in law, when $J$ goes to infinity.
So the following theorem gives an estimation of the order of convergence of $ E[f(\tilde{Z}_{n})-f(Z_{n})]$ for $ f \in \mathscr{C}²(\mathscr{I})$.
\begin{theorem}
Let denote  $T_{t}f(x,s)=E_{x,s}[f(Z_{t})]$ and assume $\forall s$ and $\forall f \in \mathscr{C}²(\mathscr{I})$, $T_{t}f(.,s)$ is in $\mathscr{C}²(\mathscr{I})$ .
Let $f \in C^{4}(\mathscr{I}) $  then it exists a function $\Tilde{\Gamma}$ at most exponential in time and a function $k_0$ linear in time which verifies when $J$ goes to infinity: there exists $\Tilde\Gamma$, $k_0$ such that
\begin{eqnarray*}
\Vert S_nf-T_nf \Vert_{J} &\le&\frac{\Tilde{\Gamma}}{J}\Big(\Vert f^{(1)} \Vert_{J}+\Vert f^{(2)} \Vert_{J}\Big)+k_0\max\limits_{s,s' \in E}\Big|J^2P_{s,s'}-\alpha_{s}Q_{s,s'}\Big|\Vert f\Vert_{J}\\
&&+O(\frac{1}{J^2})
\end{eqnarray*}

\end{theorem}

\begin{proof}
The sheme of proof will be the same than for constant selection. Let us focus on the first lemma, where some changes have to be highlighted.

\begin{lemma}
There exist bounded functions of $(x,s)$,  $\Gamma_{j}^{J}$ ($j=1,2$), and a constant $K$ such that :
\begin{align*}
\Vert(S_1-T_1)f\Vert_{J}&\leq\Vert\Gamma_1^{J}\Vert\, \Vert f^{(1)} \Vert_{J} +\Vert\Gamma_2^{J}\Vert\, \Vert f^{(2)} \Vert_{J}\\&+\sum\limits_{s'\in E} 
\left|P^{J}_{s,s'}-\frac{\alpha_{s}}{J^2}Q_{s,s'}\right|\Vert f(x,s')-f(x,s) \Vert_{J}+\frac{K}{J^3}
\end{align*}
\end{lemma}

\begin{proof}
 We provide first the quivalent of \eqref{equation11} in our context, i.e. there exists $|w'_2|<1$ such that
\begin{align*}
(S_1-T_1)f(x,s)= &\mathbb{E}_{x,s} \Big[f(X_1,s_1)\Big]-f(x,s)- L_{x,s}f(x,s)+w'_{2}L_{x,s}²f(x,s)\\
=&\mathbb{E}_{x,s} \Big[f(X_1,s_1)-f(X_{1},s)\Big]+\mathbb{E}_{x} \Big[f(X_1,s)-f(x,s)\Big]\\
&-L_{x}f(x,s)-\sum\limits_{s'\in E} \frac{\alpha_{s}}{J^2}Q_{s,s'}\left( f(x,s')-f(x,s) \right) +w'_{2}L_{x,s}^2f(x,s)
\end{align*}

In fact as before, $L_{x,s}^2f$ is still of the order of $\frac{1}{J^4}$, then

\begin{align}\nonumber
&\Big|(S_1-T_1)f(x,s)\Big|\\
&\leqslant\Big|\mathbb{E}_{x} \Big[\mathbb{E}_{s}\big[f(x_1,s_1)-f(x_{1},s)|X_{1}=x_1\big]\Big]-L_{x}f(x,s)\\ \nonumber
&\qquad+\mathbb{E}_{x} \Big[f(X_1,s)-f(x,s)\Big]
-\sum\limits_{s'\in E} \frac{\alpha_{s}}{J^2}Q_{s,s'}\left( f(x,s')-f(x,s) \right) \Big|+\frac{K}{J^4}\\ \nonumber
&\leqslant \Big|\mathbb{E}_{x} \left[ \sum\limits_{s'\in E} P_{s,s'}\big(f(X_{1},s')-f(X_{1},s)\big)-\frac{\alpha_{s}}{J^2}Q_{s,s'} \big(f(x,s')-f(x,s)\big)  \right]\\ \nonumber
&\qquad+\mathbb{E}_{x} \Big[f(X_1,s)-f(x,s)\Big]
-L_{x}f(x,s)\Big|+\frac{K}{J^4}\\ \nonumber
&\leqslant \Big|\mathbb{E}_{x} \left[ \sum\limits_{s'\in E} P_{s,s'}\big(f(X_{1},s')-f(x,s')\big)+f(x,s')\Big(P_{s,s'}-\frac{\alpha_{s}}{J^2}Q_{s,s'} \Big)  \right]\Big|\\ \nonumber
&\qquad+\Big|-P_{s,s'}f(X_{1},s)+\frac{\alpha_{s}}{J^2}Q_{s,s'}f(x,s)\Big|\\\nonumber
&\qquad+\Big|\mathbb{E}_{x} \Big[f(X_1,s)-f(x,s)\Big] 
-L_{x}f(x,s)\Big|+\frac{K}{J^4}\\ 
&\leqslant\Big|\mathbb{E}_{x} \Big[f(X_1,s)-f(x,s)\Big]
-L_{x}f(x,s)\label{s1t11}\Big| \\ 
&\qquad+ \Big|\sum\limits_{s'\in E} P_{s,s'}\left(\mathbb{E}_{x}\big[f(X_{1},s')-f(x,s')\big]+\mathbb{E}_{x}\big[f(x,s)-f(X_{1},s)\big]\right) \label{s1t12}\Big|\\
&\qquad+\sum\limits_{s'\in E} \Big|P_{s,s'}-\frac{\alpha_{s}}{J^2}Q_{s,s'}\Big|\Vert f(x,s')-f(x,s)\Vert_{J})+\frac{K}{J^4} .\label{s1t13}
\end{align}

Let now look at the order in $J$ of each term of the previous inequality. First  with the arguments used in \eqref{gamma}, there exist $K_4$ constant ,  $\Gamma^J_1$ and $\Gamma^J_2$ of the order of $\frac{1}{J^3}$ such as :
$$\mathbb{E}_{x} \Big[f(X_1,s)-f(x,s)\Big]-L_{x}f(x,s)\leqslant \Gamma_1^{J}\partial_xf (x,s) +\Gamma_2^{J} \partial_{xx} f(x,s)+\frac{K_{4}}{J^4}. 
$$
Then recall that $P_{s,s'}$ is of the order of $\frac{1}{J^2}$ and by the same calculations than in \eqref{esperance1} $\mathbb{E}_{x}\big[f(X_{1},s')-f(x,s')\big]$ is also of the order of $\frac{1}{J^2}$  so \eqref{s1t12} is at most of the order of $\frac{1}{J^4}$. \\
Finally \eqref{s1t13} can be written
$$\sum\limits_{s'\in E} \frac{1}{J^2}\Big(J^2P_{s,s'}-\alpha_{s}Q_{s,s'}\Big)\Big(f(x,s')-f(x,s)\Big)$$ and by \eqref{assumption1} is at least $o(\frac{1}{J^2})$. \\

Note that in the case where $f$ is Lipschitz in the second variable, as $s$ is of the order of $\frac{1}{J}$,  it's possible to obtain a better order $o(\frac{1}{J^3})$.

Anyway,
\begin{align*}
(S_1-T_1)f(x,s)\leqslant &\Vert\Gamma_1^{J}\Vert \Vert\partial_xf \Vert +\Vert\Gamma_2^{J} \Vert \Vert \partial_{xx} f\Vert\\
&+\sum\limits_{s'\in E} \frac{1}{J^2}\Big|J^2P_{s,s'}-\alpha_{s}Q_{s,s'}\Big|\Vert f(x,s')-f(x,s)\Vert_{J}\\
&+\frac{K'}{J^4}.
\end{align*}
\end{proof}

Assume now that  $ \forall s, f(.,s)  \in \mathscr{C}^2(I) $, and note $f^{(j)}$ the the jth derivative in $x$ of $f$.
 Note that the lemma \ref{lemme2} holds even if $s$ is no longer constant. Indeed $L_s$ is not affected by the derivative in $x$. So we get $ \forall j\in\{1,2\}$and $\forall k\leqslant j$, that there exist $ c'_j$and $a'_{k,j}\in\mathbb{R}^{+}$ independent of $J$ such that :
 $$ \Vert (T_tf)^{(j)} \Vert \leq  \exp\left(c'_j\frac{t}{J^2} \right)\sum\limits_{k=1}^{j} a'_{k,j}\Vert f^{(j)} \Vert $$
 with $ c'_j=\sup\limits_{x\in[0,1]}|j(j-1)-Js(1-2x)-Jm| $.\\
We still have that $T_t(.,s):\mathscr{C}^{2}(I)\rightarrow \mathscr{C}^2(I)$,$ \forall s$. And there exists a continuous function $R_{j}$ at most exponential in time and a linear function of time $k_0$ independent of $J$ verifying:
\begin{align*}
\Vert S_nf-T_nf \Vert_{J} &\leq \sum_{k=0}^{n-1}(\Vert (S_1-T_1)T_kf \Vert_{J} \\
&\leq \sum_{k=0}^{n-1}\Big(\Vert\Gamma_1^{J}\Vert \Vert (T_kf)^{(1)} \Vert_{J} +\Vert\Gamma_2^{J} \Vert\Vert (T_kf)^{(2)} \Vert_{J}\\
&\qquad\qquad+\sum\limits_{s'\in E} \frac{1}{J^2}\Big|J^2P_{s,s'}-\alpha_{s}Q_{s,s'}\Big|\Vert T_kf(x,s')-T_kf(x,s)\Vert_{J}\\&\qquad\qquad+O\left(\frac{1}{J^4}\right) \Big) \\
& \leq \sum_{j=1}^{2}\Vert\Gamma^{J}_j \Vert\sum\limits_{k=1}^{j} a'_{k,j}\Vert f^{(j)} \Vert_{J}\sum_{k=0}^{n-1}\exp\left(c'_j\frac{k}{J^2}\right)+O\left(\frac{1}{J^2}\right) \\
&\quad+k_0\max\limits_{s,s' \in E}\Big|J^2P_{s,s'}-\alpha_{s}Q_{s,s'}\Big|\Vert f\Vert\\
&\leq  \sum_{j=1}^{2}\Vert\Gamma^{J}_j\Vert R_{j}J^2 \Vert f^{(j)} \Vert_{J}+k_0\max\limits_{s,s' \in E}\Big|J^2P_{s,s'}-\alpha_{s}Q_{s,s'}\Big|\Vert f\Vert\\
&\quad+O\left(\frac{1}{J^2}\right )  \underset{J\to+\infty}{\longrightarrow}0
\end{align*}because if $J$ is big enough,
\begin{align*}
\sum\limits_{k=1}^{j} a'_{k,j}\sum_{k=0}^{n-1}\exp\left(c'_j\frac{k}{J^2}\right) =\sum\limits_{k=1}^{j} a'_{k,j}\frac{1-\exp(c'_{j}\frac{n}{J^2})}{1-\exp(\frac{c'_j}{J^2})}\leq J^2R_{j}(t)
\end{align*}
and $R_{j}$ is independent of $J$ because $n$ is  of the order of $J^2$. Finally, let $\Tilde{\Gamma}= \sup\limits_{J\in \mathbb{N}}\max\limits_{j\in{1,2}}\Vert\Gamma^{J}_j\Vert R_{j}J^3 $,  so that $\Tilde{\Gamma}$ does not depend on $J$ and is at most exponential in time. Then 

\begin{align*}
\Vert S_nf-T_nf \Vert_{J} &\leq \frac{\Tilde{\Gamma}}{J}\Big(\Vert f^{(1)} \Vert_{J}+\Vert f^{(2)} \Vert_{J}\Big)+k_0\max\limits_{s,s' \in E}\Big|J^2P_{s,s'}-\alpha_{s}Q_{s,s'}\Big|\Vert f\Vert_{J}\\
&\qquad+o(\frac{1}{J^2}).
\end{align*}
And this concludes the proof.
\end{proof}

\section{Random limiting selection as a diffusion process}

In this section, we assume that the limiting selection is an homogeneous diffusion process. Once again for simplicity we will  suppose that the immigration coefficient is constant. First consider the following the stochastic differential equation:

\begin{align*}
dS_t&=\frac{1}{J^2}b(S_t)dt+\frac{\sqrt{2}}{J}\sigma(S_t)dB_t\\
S_0&=s
\end{align*}
with $b$ and $  \sigma $ are both bounded and lipschitzian functions, i.e.: $\forall t\geq 0 $,$ s,s' \in \mathbb{R} $,  it exists $ k \geq 0 $ such that :
$$\vert b(t,s)-b(t,s') \vert + \vert \sigma(t,s)-\sigma(t,s') \vert \leq  L \vert s-s' \vert$$
for some constant $L$. These assumptions guarantee the existence of strong solutions of $(S_t)_{t\geq 0}$ and $(S_t)_{t\geq 0} $ has for generator
$$L_s=\frac{\sigma^2(s)}{J^2}\frac{\partial ^2}{\partial s^2}+ \frac{b(s)}{J^2} \frac{\partial}{\partial s}.$$

Let $Z_t=S_{tJ^2}$,  then the process $(S_{tJ^2})_{t \geq 0 }) $ is independent of $J$.
$$dZ_t=\sqrt{2}\sigma(Z_t)dB_t+b(Z_t)dt,\qquad Z_0=s.$$

For $T\in \mathbb{N}$,  let divide the interval $[0,T]$ in $TJ^2$ regular intervals and let introduce $\mathscr{T}_J=\{0,\frac{1}{J^2},\cdots, T \}$.
Use now the standard Euler discretization and consider
$Z_t$ defined by the relation:
$$Z_{k+1}=Z_{k}+\frac{1}{J^2}b(Z_{k})+\sqrt{2}\sigma(Z_{k})(B_{k+1}-B_k)$$
where the quantity $(B_{k+1}-B_k)_{k\leq J^2T}$ are i.i.d and follow a $\mathscr{N} (0,\frac{1}{J^2})$. It is well known that $$\sup\limits_{t\in \mathscr{T}_J}\vert \Tilde{Z}_t-Z_{tJ^2}\vert\xrightarrow{J\rightarrow\infty}0.$$
So it follows $$\sup\limits_{t\in \{0,\cdots,J^2T\}}\vert S_t-Z_{t}\vert\xrightarrow{J\rightarrow\infty}0$$

It is of course possible to use another discretization to approach $S_t$ and the following method  will still hold. There is however a small issue: in the model described in first part, for rescaling argument, the selection parameters must be in $]-1,\infty[$.  Our Markov process  $(S_t)_{t\geq 0}$ is in $\mathbb{R}$. \\
It is thus necessary to introduce the function  $h: \mathbb{R} \longrightarrow E_s$ where $E_s $ is a close bounded interval included in $]-1+\varepsilon, \infty [ $ for some $\varepsilon>0$. \\
We assume $h$ is  in $\mathscr{C}^2$ and we consider now $h(  (S_t))_{t\geq 0}$ for the selection parameter.\\
Note that to have a non trivial stochastic part in our final equation, we need as in the first section that  $h$ is of the order of $\frac{1}{J}$. Many choices are possible for $h$ and will depend on modelisation issue.

Let us give back the definition of our Moran process in this context. 
\begin{align*}
&P_+=(1-x)\left( mp+(1-m)\frac{x(1+h(s))}{1+h(s)x} \right)\\
&P_-=x\left( m(1-p)+(1-m) \left(1-\frac{x(1+h(s))}{1+h(s)x}\right) \right)\\
\end{align*}
Its first moments are given by, still denoting $\Delta=J^{-1}$,
\begin{align*}
\mathbb{E}_{x} \left[X_{n+1}-x |U_n\right]&=\Delta\Big[m(p-x)+\frac{(1-m)h(s)x(1-x)}{1+h(s)x}\Big]\\
Var\left( X_{n+1}-x|U_n\right)&=\Delta^2\Big[ \mathbb{E}_{x} \left( X_{n+1}-x \right)^2-\mathbb{E}_{x}^2  \left(X_{n+1}-x \right)\Big]\\
&=\Delta^2\Big[mp(1-2x)+x +\frac{(1-m)h(s)x(1-x)(1-2x)}{1+h(s)x}\\
&\qquad\qquad-\left(m(p-x)+\frac{(1-m)h(s)x(1-x)}{1+h(s)x}\right)^2 \Big].
\end{align*}

As in the previous case we use the process $ \left( Y_{t}\right)_{t \geq 0 }$ having the following generator to approach the Moran process when $J$ tends to infinity:

\begin{align*}
L_x=\frac{x(1-x)}{J^2}\frac{\partial^2}{\partial x^2}+\frac{1}{J}[m(p-x)+h(s)x(1-x)]\frac{\partial }{\partial x}
\end{align*}
So our aim is to give an upper bound for the error committed when $$ \left( X_n,Z_n\right)  \underset{J \to  \infty } \longrightarrow \left( Y_t,S_t \right).$$\\
Let denote by  $H$ the generator of the two dimensional  process 
$ \left( Y_t,S_t \right) $.\\

$$H=L_x+L_s=\frac{x(1-x)}{J^2}\frac{\partial^2}{\partial x^2}+\frac{1}{J^2}\Big[h(s)x(1-x)+m(p-x)\Big]\frac{\partial}{\partial x}+\frac{\sigma^2(s)}{J^2}\frac{\partial ^2}{\partial s^2}+\frac{b(s)}{J^2}\frac{\partial }{\partial s}$$

Let now state the main result of this section:
\begin{theorem}
Let $f$ be in $\mathscr{C}^4$ then there exists a function at most exponential in time $q'$ such that 
$$\sup\limits_{x\in[0,1]}\vert \mathbb{E}_{x,s}\left (f( X_n,Z_n) \right) - \mathbb{E}_{x,s}\left (f( Y_t,S_t) \right) \vert  \le\frac{q'(n)}{J}(\|\nabla f\|+\|{\rm Hess} f\|) +O(\frac{1}{J^2}).
$$
\end{theorem}
\begin{proof}
Let $ P_n $ be the operator defined on the space of bounded functions on $E$ by:

$$\left(P_nf \right)(x,s)=\mathbb{E}_{x,s}\Big[f\left(X_n,Z_n \right)\Big]$$ 
It is of course a semigroup so that $P_{n+m}=P_nP_m  ,\forall m,n \in  \mathbb{N}$. In parallel, let $ (T_t)_{t \geq 0 } $ be defined on the space of bounded continuous functions by :

$$\left( T_tf \right)(x,s)=\mathbb{E}_{x,s} \Big[ f \left( Y_t,S_t \right)  \Big] $$ 
also verifying,$T_{t+s}=T_sT_t, \forall s\geq 0,t \geq 0.$ The starting point is as in the first part of \eqref{equation11},
$$
\vert P_nf(x,s)-T_nf(x,s) \vert \leq \sum_{k=0}^{n-1} \Vert \left( P_1-T_1 \right) T_kf \Vert.
$$

We now focus on the quantity $\Vert \left( P_1-T_1 \right) T_kf \Vert$, the following lemma gives a upper bound of the quantity $\Vert \left( P_1-T_1 \right) f \Vert$ for $f$ in $\mathscr{C}⁴$.

\begin{lemma}
Let $f$ be in  $\mathscr{C}⁴$ it exists
  $ \gamma^J_{1} $ ,and $\gamma_{2}^J$ such as : 
\begin{align*}
\Vert (P_1f)(x,s)-(T_1f)(x,s) \Vert =& \gamma^J_{1}  \Vert \frac{\partial}{\partial x}  f(x,s) \Vert + \gamma^J_{2}  \Vert \frac{\partial^2}{\partial x^2} f(x,s) \Vert +O(\frac{1}{J^4})
\end{align*} 
where $\gamma^J_1$ and $\gamma^J_2$ are of order $\frac1{J^3}$.
\end{lemma}

\begin{proof}
We will use the same methodology. First the Taylor expansion (in space) of $P_1$ gives:

\begin{align*}
(P_1f)(x,s)=&f(x,s)+\mathbb{E}_{x,s}\Big[X_1-x \Big]\frac{\partial}{\partial x}f(x,s)+\mathbb{E}_{x,s}\Big[Z_1-s \Big]\frac{\partial}{\partial s}f(x,s)\\
&+\frac{1}{2}\mathbb{E}_{x,s}\Big[(X_1-x )^2	\Big]\frac{\partial ^2}{\partial x^2}f(x,s)+\frac{1}{2}\mathbb{E}_{x,s}\Big[(Z_1-s)^2 \Big]\frac{\partial^2}{\partial s^2}f(x,s)\\
&+2\mathbb{E}_{x,s}\Big[(X_1-x)(Z_1-s) \Big]\frac{\partial^2}{\partial x \partial s}f(x,s)+O(\frac{1}{J^4}).
\end{align*}
Indeed we have the quantities:
\begin{align*}
&\mathbb{E}_{x,s}\Big[X_1-x \Big]=\frac{1}{J}\Big[m(p-x)+\frac{(1-m)h(s)x(1-x)}{1+h(s)x}\Big]\\
&\mathbb{E}_{x,s}\Big[(X_1-x)^2 \Big]=\frac{1}{J^2}\Big[mp(1-2x)+x+\frac{(1-m)x(1+h(s))(1-2x)}{1+h(s)x}\Big]\\
&\mathbb{E}_{x,s}\Big[Z_1-s \Big]=\frac{b(s)}{J^2}\\
&\mathbb{E}_{x,s}\Big[(Z_1-s)^2 \Big]=\frac{b^2(s)}{J^4}+\frac{2\sigma^2(s)}{J^2}=\frac{2\sigma^2(s)}{J^2}+O(\frac{1}{J^4})\\
&\mathbb{E}_{x,s}\Big[(X_1-x)(Z_1-s) \Big]=\mathbb{E}_{x,s}\Big[X_1-x \Big]\mathbb{E}_{x,s}\Big[Z_1-s \Big]\\
&\qquad \qquad \qquad \qquad \qquad = \frac{b(s)}{J^3}\Big[m(p-x)+\frac{(1-m)h(s)x(1-x)}{1+h(s)x}\Big]=O(\frac{1}{J^4})\\
&\mathbb{E}_{x,s}\Big[(X_1-x)^3 \Big]=O(\frac{1}{J^4}),\qquad\qquad\mathbb{E}_{x,s}\Big[(Z_1-s)^3 \Big]=O(\frac{1}{J^4}).
\end{align*}
And the Taylor expansion of $T_1$ in times gives:
$$\left(T_1f \right)(x,s)=f(x,s)+ L_xf(x,s)+L_sf(x,s)+O(\frac{1}{J^4}).$$
Indeed it is easy to see that $H²$ is $O(\frac{1}{J^4})$. Do now the difference
\begin{align*}
(P_1-T_1)f(x,s)=&\mathbb{E}_{x,s}\Big[X_1-x \Big]\frac{\partial}{\partial x}f(x,s)+\mathbb{E}_{x,s}\Big[Z_1-s \Big]\frac{\partial}{\partial s}f(x,s)\\
&+\frac{1}{2}\mathbb{E}_{x,s}\Big[(X_1-x )^2	\Big]\frac{\partial ^2}{\partial x^2}f(x,s)\\
&+\frac{1}{2}\mathbb{E}_{x,s}\Big[(Z_1-s)^2 \Big]\frac{\partial^2}{\partial s^2}f(x,s)\\
&- L_xf(x,s)-L_sf(x,s) +O(\frac{1}{J^4}).
\end{align*}
Finally, \begin{align*}
&(P_1-T_1)f(x,s)\\
&\qquad=-\frac{1}{J}\Big[\frac{h(s)x(1-x)}{1+h(s)x}(m+h(s)x)\Big]\frac{\partial }{\partial x}f(x,s)\\
&\qquad\quad+\frac{h(s)x²+x(1-2x)(h(s)-m-mh(s))-2x(1-x)h(s)}{1+h(s)x}\frac{\partial ^2}{\partial x^2}f(x,s)\\
 &\qquad\quad+\frac{b^2(s)}{J^4}\frac{\partial^2}{\partial s^2}f(x,s)+O(\frac{1}{J^4}).
\end{align*}
Let us conclude by taking the norm to get
\begin{align*}
&\gamma^J_{1}=\sup\limits_{(x,s)\in E_s\times[0,1]}\frac{1}{J}\Big|\frac{h(s)x(1-x)}{1+h(s)x}(m+h(s)x)\Big|\\
&\gamma^J_{2}=\sup\limits_{(x,s)\in E_s\times[0,1]}\frac{1}{2J^2}\Big|mp(1-2x)+ \frac{(1-m)(1+h(s))x(1-2x)}{1+h(s)x}-x(1-2x) \Big|
\end{align*}so that we obtain the result.
\end{proof}

Then \eqref{assumption1} still holds for this case as the proof of \ref{prop1} is exactly the same, so the end follows as in the first part.
\end{proof}

\section{Appendices : Wright-Fisher discrete model and its approximation diffusion}

Let's consider the Wright-Fisher discrete model with selection and immigration. The population  still consists of two species, immigration and selection are still the same. But the Markovian process $X^n_J$ evolves according to the following probability:   

\begin{equation*}
  \pp\left(X^J_{n+1}=\frac{k}{J}|X^J_{n}=x\right)  = \left(\begin{array}{c} J\\k\end{array}\right)P_x^k(1-P_x)^{J-k}
\end{equation*}
\vspace{2mm}
with $P_x=mp+(1-m)\frac{(1+s)x}{1+sx}$.

At each step, all the population is renewed, so this process goes $J$ times faster than the Moran process. And we usually, in the case of weak selection and immigration, use the diffusion $\{Y_t\}_{t>0}$ defined by the following generator to approach this discrete model, when the population goes to infinity.

$$ L=\frac{1}{2J}x(1-x)\frac{{\partial}^2}{\partial{x}^2}+ \left( sx(1-x)+m(p-x)\right)\frac{\partial}{\partial{x}}$$

\begin{theorem}
\quad\\
Let $f$ be in $C^5(I)$ then there is a function $q(t)$ growing at most exponentially in time, depending on $m'$ and $s'$ which satisfies when $J$ goes to infinity:
$$
\sup\limits_{x\in I_J}\vert \mathbb{E}_x \Big[f(X_{n}^{J} )\Big]-\mathbb{E}_x \Big[f(Y_{n}^{J} ) \Big] \vert \le \left(\Vert f^{(1)} \Vert_{J} +\Vert f^{(2)} \Vert_{J}+\Vert f^{(3)} \Vert_{J}\right)\frac{q(n)}{J}+o\left(\frac{1}{J}\right).
$$
\end{theorem}

\begin{proof}
Even if the structure of the proof is the same than for the Moran model, however the difference of scale (in $\frac{1}{J}$ now) causes some small differences. Mainly,  the calculation of the $\{\gamma_j\}_{j \in \{1,2,3,4\}}$ is a bit different.
Note that we need to have $f \in C^5$ in the previous theorem, which is stronger than for the Moran process. The main explanation comes from the calculation of $E[(X_{n+1}^J-x)^k|X_{n}^J=x]$, for which for the Wright-Fisher discrete process it is no longer of the order of $\frac{1}{J^k}$. Let us give some details.

First consider the moments $\{ E[\big(X_{n+1}^J-x\big)^k|X_n=x]\}_{k\leqslant 5}$:

\begin{align*}
E[X_{n+1}^J-x|X_n=x]=&m(p-x)+\frac{sx(1-x)}{1+sx}\\
E[\big(X_{n+1}^J-x\big)^2|X_n=x]=&\frac{1}{J}x(1-x)+\frac{1}{J}\left(m(p-x)+\frac{sx(1-x)}{1+sx}\right)\\
&+\left(m(p-x)+\frac{sx(1-x)}{1+sx}\right)^2+O(\frac{1}{J^3})\\
E[\big(X_{n+1}^J-x\big)^3|X_n=x]=&x(x-1)(2x-1)\frac{1}{J^2}\\
&-\frac{1}{J}3x(x-1)\big(m(p-x)+sx(1-x)\big)+O(\frac{1}{J^3})\\
E[\big(X_{n+1}^J-x\big)^4|X_n=x]=&\frac{1}{J^2}3x^2(1-x)^2+O(\frac{1}{J^3})\\
E[\big(X_{n+1}^J-x\big)^5|X_n=x]=& O(\frac{1}{J^3}).
\end{align*}
To get a quantity of the order of $\frac{1}{J^3}$ we need to go to the fifth moment of $X_{n+1}^J-x$, so in the Taylor development we need to have $f$ in $C^5$. Then, 

\begin{align*}
L^{1}f(x)=&\frac{x(1-x)}{2J}f^{(2)}(x)+sx(1-x)+m(p-x)f^{(1)}(x)\\
L^{2}f(x)=&\frac{(x(1-x))^2}{4J^2}f^{(4)}(x)\\
&+ \Big[ \frac{2(1-2x)x(1-x)}{4J^2}+\frac{2x(1-x)(sx(1-x)+m(p-x)}{2J}\Big]f^{(3)}(x)\\
&+ \Big[   \frac{-2x(1-x)}{4J^2} \\
&\quad+\frac{2x(1-x)(s(1-2x)-m)+(1-2x)(sx(1-x)+m(p-x))}{2J}\\
&\quad+\frac{(sx(1-x)+m(p-x))^2}{4J^2}\Big]f^{(2)}(x) \\
&+ \Big[\frac{-2sx(1-x)}{2J}\\
&\quad+(s(1-2x)-m)(sx(1-x)+m(p-x))\Big]f^{(1)}(x)\\
L^{3}f(x)&=O(\frac{1}{J^3}).
\end{align*}

We are now able to give the expression of the $\{\gamma_j\}_{j \in \{1,2,3,4\}}$, as in the lemma $\ref{lemme1}$.

\begin{lemma}
\quad\\
It exists bounded functions of $x$,  $\{\gamma_j\}_{j \in \{1,2,3\}}$ such as when J is big enough, $$ \Vert(S_1-T_1)f\Vert_{J}\leq\Vert\gamma_1^{J}\Vert \Vert f^{(1)} \Vert_{J} +\Vert\gamma_2^{J}\Vert \Vert f^{(2)} \Vert_{J}+\Vert\gamma_3^{J}\Vert \Vert f^{(3)} \Vert_{J}+\frac{K_{1}}{J^3}$$
\label{lemmeWF}
where for $i=1,...,3$, $\Vert\gamma_i^{J}\Vert\sim\frac1{J^2}$.
\end{lemma}

\begin{proof}
The proof of this lemma is exactly the same than in lemma \ref{lemme1}.
Just the calculations are a little bit more tedious:

\begin{align*}
&\gamma_1^J=\frac{-sx(1-x)}{J}
+(s(1-2x)-m)(sx(1-x)+m(p-x))\\
&\gamma_2^J= \frac{-x(1-x)}{4J^2} 
+\frac{xs(6x^2-7x+1)+m(4x^2-2xp-x-p)}{4J}+O(\frac{1}{J^3})\\
&\gamma_3^J=\frac{x(x-1)(2x-1)}{12J^2}+O(\frac{1}{J^3})\\
&\gamma_4^J=O(\frac{1}{J^3})
\end{align*}
\end{proof}
The end of the proof follow exactly the same pattern.

\end{proof}
So  the Wright-Fisher dynamics causes harder calculations than the Moran model but the spirit of the proof is the same.
So All the methods studies in this paper still hold for the Wright-Fisher model.

\bibliographystyle{plain}
\bibliography{biblio}

\end{document}